\documentclass{article}

\usepackage[utf8]{inputenc}
\usepackage[T1]{fontenc}

\usepackage{amssymb,amsthm}
\usepackage[centertags]{amsmath}
\usepackage{MnSymbol}
\usepackage[mathscr]{euscript}
\usepackage{tikz-cd}
\usepackage{hyperref}
\hypersetup{
    colorlinks=true,
    linkcolor=[HTML]{0D47A1},		
    citecolor=[HTML]{0D47A1},		
    urlcolor=[HTML]{0D47A1}		    
}
\usepackage{fullpage}
\usepackage{rotating}
\usepackage[capitalize,nameinlink]{cleveref}

\usepackage{enumitem}
\usepackage{bbm}

\usepackage[shortcuts]{extdash}

\usepackage{tabularx}
\usepackage{multirow}
\usepackage{float}

\definecolor{mat}{HTML}{ffd6ad}

\definecolor{jon}{HTML}{7fdfff}

\newtheorem{thm}{Theorem}[subsection]

\newtheorem{lemma}[thm]{Lemma}
\newtheorem{proposition}[thm]{Proposition}

\theoremstyle{definition}

\newtheorem{definition}[thm]{Definition}

\newtheorem{remark}[thm]{Remark}

\newcommand\Ra{\Rightarrow}

\newcommand\slice[1]{_{/#1}}

\newcommand\pbmark{\ar[dr, phantom, "\ulcorner" very near start, shift right=1ex]}

\newcommand\Sh[1]{\mathrm{Sh}\!\left(#1\right)}
\newcommand\iSh[1]{\mathrm{Sh}_\infty\!\left(#1\right)}

\newcommand\Op[1]{\mathrm{Op}\left(#1\right)}

\newcommand\cB{\mathscr{B}}
\newcommand\cC{\mathscr{C}}
\newcommand\cD{\mathscr{D}}
\newcommand\cE{\mathscr{E}}
\newcommand\cF{\mathscr{F}}

\newcommand\cL{\mathscr{L}}

\newcommand\cO{\mathscr{O}}

\newcommand\cR{\mathscr{R}}
\newcommand\cS{\mathscr{S}}
\newcommand\cT{\mathscr{T}}
\newcommand\cU{\mathscr{U}}

\newcommand\cW{\mathscr{W}}

\newcommand\BB{\mathbb{B}}
\newcommand\CC{\mathbb{C}}
\newcommand\DD{\mathbb{D}}
\newcommand\EE{\mathbb{E}}
\newcommand\FF{\mathbb{F}}

\newcommand\LL{\mathbb{L}}

\newcommand\RR{\mathbb{R}}

\newcommand\TT{\mathbb{T}}
\newcommand\UU{\mathbb{U}}

\newcommand\WW{\mathbb{W}}

\newcommand{\tA}{\mathsf A}

\newcommand{\tX}{\mathsf X}

\newcommand\xto\xrightarrow
\newcommand\xot\xleftarrow
\newcommand\ot\leftarrow
\newcommand\Alpha{\mathrm{A}}

\newcommand\fun[2]{\left[#1,#2\right]}

\newcommand\oo{$\infty$\=/}

\newcommand\Arr[1]{{#1}^{\to}}

\newcommand\Locale{\mathsf{Locale}}
\newcommand\Poset{\mathsf{Poset}}

\newcommand\Set{\mathsf{Set}}
\newcommand\SET{\mathsf{SET}}

\newcommand\Cat{\mathsf{Cat}}
\newcommand\CCat{\mathsf{\CC at}}
\newcommand\CAT{\mathsf{CAT}}

\newcommand\Topos{\mathsf{Topos}}

\newcommand\LFib{\mathsf{LFib}}

\newcommand\CcFib{\mathsf{CcFib}}

\newcommand\Fam{\mathsf{Fam}}

\newcommand\op{^\mathrm{op}}
\newcommand\sfop{^\mathsf{op}}

\newcommand\size[1]{^{<#1}}

\title{
Smooth and proper maps \\
with respect to a fibration
\footnote{This material is based upon work supported by the Air Force Office of Scientific Research under award number FA9550-20-1-0305 and by the US Army Research Office under MURI Grant W911NF-20-1-0082 which the second author was hosted under at the Department of Mathematics at Johns Hopkins University. 
The second author wishes to thank the Max Planck Institute for Mathematics in Bonn for its hospitality and financial support during some stages of this project.}
}

\author{Mathieu Anel
\footnote{
Department of Philosophy, 
Carnegie Mellon University,
\href{mailto:mathieu.anel@protonmail.com}{mathieu.anel@protonmail.com}
}
\and
Jonathan Weinberger
\footnote{
Fowler School of Engineering \& Center of Excellence in Computation, Algebra, and Topology (CECAT),
Chapman University, 
\href{mailto:jweinberger@chapman.edu}{jweinberger@chapman.edu}
}
}

\begin{document}

\maketitle

\begin{center}
{\it To André Joyal, on his $(\infty,0)$-th Birthday}
\end{center}

\begin{abstract}
This paper explain how the geometric notions of local contractibility and properness are related to the $\Sigma$-types and $\Pi$-types constructors of dependent type theory.
We shall see how every Grothendieck fibration comes canonically with such a pair of notions---called smooth and proper maps---and how this recovers the previous examples and many more.
This paper uses category theory to reveal a common structure between geometry and logic, with the hope that the parallel will be beneficial to both fields.
The style is mostly expository, the main results are proved in external references.
\end{abstract}

\tableofcontents

\section{Introduction}



Dependent type theory is based on the notion of family of types indexed by a type, and the basic operations are the reindexing, and the sums and products of such families: the $\Sigma$-types and $\Pi$-types constructors. 
On the other hand, the geometer's toolbox contains methods to study spaces by means of ``bundles,'' that is using families of spaces indexed by the space of interest (vector bundles, sheaves\ldots).
There also, bundles can be pulled back along a morphism and sometimes pushed forward in two different ways (additively or multiplicatively).
This has suggested an analogy where $\Sigma$-types can be understood as total spaces of families, while $\Pi$-types can be regarded as spaces of sections of families.

The fact that the pushforwards are not always defined is the source of an interesting feature: it distinguishes the classes of maps along which these pushforwards exists: locally contractible and proper morphisms.
\medskip
In \cref{sec:def}, we introduce an abstract notion of \emph{smooth} and \emph{proper maps} associated to any category $\CC$.
This is done in the setting of fibered/indexed categories over a base category $\cB$.
The smooth (proper) maps are the maps $u:X\to Y$ in $\cB$ along which the base change functor $u^*:\CC(Y)\to \CC(X)$ admits a left adjoint $u_!$ (a right adjoint $u_*$) compatible with reindexing/base change (aka the Beck--Chevalley conditions).
Then \Cref{sec:ex} details many examples in logic, category theory, topology, and geometry, where we show how our abstract definitions connect to existing notions of smoothness and properness.

\paragraph{Acknowledgments}
The authors are happy to dedicate this paper to André Joyal on the occasion of his 80th birthday.
The focus on examples is an hommage to André, who has drilled their importance into the head of the first author.

The authors would like to thank the anonymous referee for valuable feedback, as well as
Carlo Angiuli,
Steve Awodey,
Reid Barton,
Denis-Charles Cisinski,
Jonas Frey,
Louis Martini,
Anders Mörtberg,
Emily Riehl,
Jon Sterling,
Thomas Streicher,
Andrew Swan,
and Sebastian Wolf
for numerous discussions around this material.
Special thanks to Denis-Charles Cisinski for his comments on an earlier version, and to Jon Sterling for pointing out to us the example of dominances.

\paragraph{Conventions}
The paper is written in the context of \oo categories \cite{Lurie:HTT,Cisinski:book}, but we are simplifying the terminology and simply say ``category'' for \oo category and  ``1-categories'' for the truncated notion.
We denote by $\Set$ ($\SET$) the 1-category of small (large) sets,
and by $\Cat$ ($\CAT$) the (\oo)category of small (large) (\oo)categories.  Consequently, we refer to \oo functors simply as ``functors,'' so that under the (\oo categorical) Grothendieck construction small fibered categories correspond to functors into $\Cat$. The categories of functors from $C$ to $D$ is denoted $\fun CD$.
The arrow category of a category $C$ is denoted $\Arr C$.

\section{Abstract setting}
\label{sec:def}
\label{setting}

Let $\kappa$ be a cardinal and $\Set\size\kappa$ be the category of sets of cardinality strictly smaller than $\kappa$.
The elementary operations on sets are the sum and product of a family of sets.
For $I$ an arbitrary set, an $I$-indexed family ($I$-family for short) in $\Set\size\kappa$ is a functor $I\to \Set\size\kappa$, the sum and product of $I$-families are the left and right adjoint to the constant family functor $\Set\size\kappa \to \fun I {\Set\size\kappa}$.
Given $\kappa$, one can ask for what sets $I$ the sum and product of $I$-families of $\kappa$-small sets are $\kappa$-small.
Let $\sigma$ be a cardinal such that for any cardinal $\rho<\kappa$ we have $\sigma\rho<\kappa$.
Then $\Set\size\kappa$ admits sums indexed by objects in $\Set\size\sigma$.
Let $\Sigma(\kappa)$ be the supremum of all such cardinals $\sigma$.
We shall call $\Sigma(\kappa)$ the {\it smooth bound} of $\kappa$.
The cardinal $\kappa$ is \emph{regular} if and only if $\kappa=\Sigma(\kappa)$.
Similarly, let $\pi$ be a cardinal such that for any cardinal $\rho<\kappa$ we have $\rho^\pi<\kappa$.
Then $\Set\size\kappa$ admits products indexed by objects in $\Set\size\pi$.
Let $\Pi(\kappa)$ be the supremum of all such cardinals $\pi$.
We shall call $\Pi(\kappa)$ the {\it proper bound} of $\kappa$.
The cardinal $\kappa$ is \emph{inaccessible} if and only if $\kappa=\Pi(\kappa)=\Sigma(\kappa)$.

More generally, we can replace the category $\Set\size\kappa$ by any category $C$ and extract the classes of sets $\Sigma(C)$ and $\Pi(C)$ indexing the sums and products which exist in $C$.
We shall call $\Sigma(C)$ the {\it smooth calibration} of $C$
and $\Pi(C)$ the {\it proper calibration} of $C$.%
\footnote{The names {\it smooth} and {\it proper} are taken from Grothendieck in Pursuing Stacks \cite{Maltsiniotis:aspheric} and the example of left fibrations of categories. The name {\it calibration} is borrowed from Bénabou \cite{Benabou:calibration}, even if his notion is slightly different.}

We are going to propose an abstract setting for the definition of smooth and proper calibrations and illustrate it with many examples.
The category $C$ will be a fibration over some base category $B$ with finite limits,
and the calibrations $\Sigma(C)$ and $\Pi(C)$ will be defined as subfibrations of the codomain fibration of $B$.

\subsection{Calibrations and families}

We fix a category $\cB$ with finite limits.

\begin{definition}[$\cB$-category]
\label{def:category}
A functor $\CC:\cB\op\to \CAT$ is be called a {\it $\cB$-category} (or simply a category when $\cB$ is clear from the context).
A ($\cB$-)functor between $\cB$-categories is simply a natural transformation.
There is an equivalence between the category of $\cB$-categories and the category of fibrations in categories (aka Grothendieck fibrations, or cartesian fibrations) $\cC\to \cB$ \cite[Theorem 3.2.0.1 and Section 3.3.2]{Lurie:HTT}.
We shall treat the two objects $\CC$ and $\cC$ as equivalent, and go back and forth between them.
We systematically use the same letter in different fonts $\CC$/$\cC$ to denote whether we are considering the functorial (aka indexed) or the fibrational point of view on $\cB$-categories.
\end{definition}

\begin{definition}[Universe]
\label{def:universe}
Since $\cB$ is assumed with finite limits, the codomain functor $\Arr\cB\to \cB$ is a cartesian fibration and we denote by $\BB:\cB\op\to \CAT$ the corresponding functor (sending an object $X$ to the slice category $\cB\slice X$).
We shall refer to $\BB$ as the {\it universe} of $\cB$.
\end{definition}

Since $\cB$ has a terminal object, the terminal functor $1:\cB\op\to \CAT$, is a subfunctor of the universe $\BB$.
The corresponding fibration is the identity of $\cB$.
Equivalently, it corresponds to the subcodomain fibration $\cB=\cB^\simeq\subset \Arr\cB$ spanned by the isomorphisms.%
\footnote{We do not distinguish equivalent categories here.}

\begin{definition}[Calibration]
\label{def:calibration}
A {\it calibration} is a subuniverse $\UU\subset \BB$ (i.e. a subfibration $\cU\subset \Arr\cB$).
The {\it constant calibration} is the calibration $1\to \BB$.
A calibration is {\it pointed} if it contains the constant calibration (equivalently, if $\cU\subset \Arr\cB$ contains all isomorphisms).
A calibration $\UU\subset \BB$ is {\it regular} if it is pointed and the corresponding class of maps in $\cB$ is closed under composition.
\end{definition}

A regular calibration defines a wide subcategory of $\cB$ (since it contains all isomorphisms).
In fact, regular calibrations are in bijection with wide subcategories of $\cB$ whose maps are closed under base change in $\cB$.

\begin{definition}[Family construction]
\label{def:family}
We fix a $\cB$-category $\CC$, with associated fibration $\cC\to \cB$. 
Let $\cU\subset \Arr\cB \xto {cod} \cB$ be a subfibration of the codomain fibration.
The fibration $\Fam_\cU(\cC)$ of {\it $\cU$-indexed families of objects in $\cC$} is defined as the functor $\phi$ in the diagram
\[
\begin{tikzcd}
\Fam_\cU(\cC)\ar[r] \ar[d]\ar[dd,bend right,"\phi"']\pbmark &\cC \ar[d]\\
\cU\ar[r,"dom"]\ar[d,"cod"]&\cB\\
\cB
\end{tikzcd}
\]
(where the square is a fiber product).
Following our typographic convention, we denote by $\Fam_\UU(\CC)$ the associated functor $\cB\op\to \CAT$ (and call it the $\cB$-category of {\it $\UU$-indexed families of objects in $\CC$}).
Its value at an object $I$ in $\cB$ is the category of pairs $(u:U\to I \in \cU,C\in \CC(U))$ (with the obvious notion of morphism).
\end{definition}

\subsection{Sums and products}

The family construction associated to the constant calibration is the identity: $\CC = \Fam_1(\CC)$.
If $\UU$ is a pointed calibration we get a canonical functor $\Delta_\UU:\CC=\Fam_1(\CC)\to \Fam_\UU(\CC)$.

\begin{definition}
\label{def:sum-product}
Let $\UU$ be a pointed calibration on $\cB$.
A ($\cB$-)category $\CC$ has $\UU$-indexed sums if the canonical functor $\Delta_\UU:\CC\to\Fam_\UU(\CC)$ has a left adjoint.
A ($\cB$-)category $\CC$ has $\UU$-indexed products if $\CC\op$ has $\UU$-indexed sums.
\end{definition}

The following result is proved for 1-categories in \cite{Streicher:fibrations}.
The statement extends to \oo categories as well. So do a range of expected results from~\cite{Streicher:fibrations,Moens:PhD} about fibrations with internal sums. This has been developed by \cite{Buchholtz-Weinberger,Weinberger:IntlSums} in the type theory of synthetic \oo categories Riehl--Shulman \cite{Riehl-Shulman,Riehl:Notices} (internally to \emph{any} $\infty$-topos, see \cite{Shulman:topos,Riehl:Semantics,Weinberger:StrExt}). See also the work on internal $\infty$-categories in an arbitrary $\infty$-topos by Martini--Wolf~\cite{Martini:cocart,Martini-Wolf:internal,Martini-Wolf:proper}.

\begin{proposition}[B\'enabou--Streicher]
\label{prop:prop-univ-fam}
If $\UU$ is a regular calibration, then $\Fam_\UU(\CC)$ is the free cocompletion of $\CC$ for sums indexed by objects in $\UU$.
A pointed calibration $\UU$ is regular if and only if $\UU$ has $\UU$-indexed sums.
\end{proposition}

\begin{remark}
The notion of regular calibration corresponds in dependent type theory to the notion of a subuniverse closed under $\Sigma$-types and containing the terminal type.
\end{remark}

\begin{proposition}
\label{prop:codomain-fib}
For every category with finite limits $\cB$, the universe $\BB$ always has $\BB$-indexed sums.
A map is proper if and only if all its pullbacks are exponentiable maps in $\cB$.
In particular, the universe $\BB$ has $\BB$-indexed products if and only if $\cB$ is a locally cartesian closed category.
\end{proposition}
\begin{proof}[Sketch of the proof]
The left adjoints $f_!$ always exists and are given by composition with $f$.
The Beck--Chevalley conditions are fulfilled since they amount to the cancellation property for pullback squares.
The right adjoint $f_*$ exists for every map $f$ if and only if $\cB$ is a locally cartesian closed category.
Then the Beck--Chevalley conditions are deduced by adjunction from those for $f_!$.
\end{proof}

\subsection{Smooth and proper maps}

We fix a category $\cB$ with finite limits and a $\cB$-category $\CC:\cB\op\to \CAT$.

\begin{definition}[Smooth \& proper maps]
\label{def:smooth}
\label{def:proper}
A map $u:X\to Y$ in $\cB$ is called {\it left ({right}) Beck--Chevalley} for $\CC$
if for any cartesian square 
\[
\begin{tikzcd}
\bar Y\ar[r,"\bar v"] \ar[d,"\bar u"']\pbmark &Y \ar[d,"u"]\\
\bar X\ar[r,"v"] &X
\end{tikzcd}
\]
the maps $u^*$ and $\bar u ^*$ have left ({right}) adjoints 
\[
\begin{tikzcd}
\CC(\bar Y)\ar[from=r,"\bar v^*"']
\ar[d, shift right=2, dashed,"\bar u_!"']
\ar[from=d]
\ar[d, shift left=2,dotted,"\bar u_*"]
& \CC(Y)
\ar[d, shift right=2, dashed,"u_!"']
\ar[from=d]
\ar[d, shift left=2,dotted,"u_*"]\\
\CC(\bar X)\ar[from=r,"v^*"']
& \CC(X)
\end{tikzcd}
\]
and the corresponding mate natural transformation is invertible
\[
{\bar u_!} \bar v^* \xto\sim v^* {u_!}
\qquad\qquad
\big(\ v^*{u_*} \xto\sim {\bar u_*}\bar v^*\ \big)\,.
\]
A map $u:X\to Y$ in $\cB$ is called {\it smooth}, or {\it stably left Beck--Chevalley}
({\it proper}, or {\it stably right Beck--Chevalley})
if every base change of $u$ is left (right) Beck--Chevalley.
The relation of these operations with quantification in logic is recalled in \cref{table:logic}.
\end{definition}

The classes of $\CC$-smooth and $\CC$-proper maps are closed under base change and define subfibrations of the codomain fibration of $\cB$.
Equivalently, they define regular calibrations $\Sigma(\CC) \subset \BB\supset \Pi(\CC)$ of the universe of $\BB$, called the \emph{smooth calibration} and the \emph{proper calibration} of $\CC$.
The interest of the notions is in the following result.

\begin{proposition}
The smooth (proper) calibration of $\CC$ is the largest calibration for which $\CC$ admits sums (products).
\end{proposition}
\begin{proof}
Unfolding the condition that $\Delta_\UU:\CC\to\Fam_\UU(\CC)$ has a left (right) adjoint, 
we find that the left adjoint exists if and only if all maps in $\cU$ are smooth (proper), see \cite{Streicher:fibrations}.
\end{proof}

\begin{remark}
\label{rem:BC}
When $\cB$ is the 1-category of sets or the \oo category of spaces, the Beck-Chevalley conditions simplify for the functors $\CC:\cB\op\to \CAT$ which send colimits to limits.
In that case, we have $\CC(X) = \CC(1)^X$ and the compatibility with base change holds for free as soon as the adjoint functors exist.
\end{remark}

The characterization of the smooth and proper maps can be quite difficult in practice.
In the setting where the category $\CC$ has a forgetful functor into the universe $\BB$ (typically, when $\CC$ classifies objects in $\BB$ with an extra structure, notably when $\CC$ is a calibration) stricter notions of smooth and proper maps can be defined which are easier to characterize in practice.

\begin{definition}[Strict smoothness/properness]
Let $\CC:\cB\op \to \CAT$ be equipped with a natural transformation $U:\CC\to\BB$ to the universe of $\BB$ (typically the inclusion of a subuniverse).
For $f:X\to Y$ a map in $\cB$, we have a commutative square
\[
\begin{tikzcd}
\CC(Y)\ar[d,"f^*"'] \ar[r,"U_Y"]& \cB\slice Y \ar[d,"f^*"]\\
\CC(X) \ar[r,"U_X"]& \cB\slice X\,.
\end{tikzcd}
\]
Recall from \cref{prop:codomain-fib} that the functor $f^*:\cB\slice Y\to \cB\slice X$ has always a left adjoint $f_!$ given by composition with $f$ and that the right adjoint $f_*$ exists if and only if $f$ is an exponentiable map in $\cB$.
We shall say that a $\CC$-smooth map is {\it strictly smooth} if the mate $f_!U_X \to U_Yf_!$ is invertible.
We shall say that a $\CC$-proper map is {\it strictly proper} if it is exponentiable in $\cB$ (\emph{i.e.}, if $f_*$ exists) and if the mate $U_Yf_* \to f_*U_X$ is invertible.
%
Essentially, this says that a map $f$ is strictly smooth if $f_!$ can be computed by composition with $f$ in $\cB$, and that it is strictly proper if $f_*$ can be computed by exponentiation along $f$ in $\cB$.
\end{definition}

\begin{remark}
We shall see that such maps are sometimes easier to characterize in practice.
In particular, when $\CC$ is a calibration, a map is strictly smooth if the corresponding class $\cC$ of maps in $\cB$ is closed under composition with every pullback of $f$.
And a map $f$ is strictly proper if it is exponentiable and the functor $f_*$ sends maps in $\cC$ to maps in $\cC$.
\end{remark}

\begin{lemma}
\label{lem:strict-smooth}
If $\CC\subset \BB$ is a regular calibration, then a $\CC$-smooth map is strict if and only if it is in $\cC\subset \Arr\cB$.
\end{lemma}
\begin{proof}
A map $u:X\to Y$ is strict if composition with $u$ sends maps $r:\bar X\to X$ in $\cC$ to maps $\bar X\to Y$ in $\cC$.
Applied to $r=id_X$, this implies that $u$ is in $\cC$ and that every strict smooth maps is in $\cC$.
The converse is true by regularity of $\CC$.
\end{proof}

A {\it modality} on the category $\cB$ is a (unique) factorization system $(\cL,\cR)$ such that the factorization (or equivalently the left class $\cL$) is stable under base change.
In this situation, both classes $\cL$ and $\cR$ define calibrations $\LL\subset\BB\supset\RR$ and $\RR$ is even a reflective calibration (with the reflection given by the factorization) \cite{ABFJ:GBM,RSS}.

\begin{lemma}
\label{lem:strict-smooth-modality}
If $\CC=\RR\subset \BB$ is the subuniverse of the right class of a modality on $\cB$, 
then every map in $\cB$ is $\CC$-smooth.
\end{lemma}
\begin{proof}
For a map $u:X\to Y$ in $\cB$, and $r:\bar X\to X$ in $\cR$, 
the left adjoint $u_!(r)$ is given by the right map of the $(\cL,\cR)$-factorization of the map composite map $\bar X\to X\to Y$.
It satisfies the Beck--Chevalley condition because the factorization is stable under base change.
\end{proof}

Notice that \cref{lem:strict-smooth,lem:strict-smooth-modality} together provide examples where smooth and strictly smooth maps do not coincide.

\medskip
We consider now the case of a {\it weak} factorization system $(\cL,\cR)$ on $\cB$.
The class $\cR$ is still closed under base change and defines a regular calibration $\RR\subset \BB$ (but not the class $\cL$ in general).
The following lemma is a classical trick to recognize proper maps.

\begin{lemma}
\label{lem:strict-proper}
If $\CC=\RR\subset \BB$ if the regular calibration associated to the right class of a weak factorization system $(\cL,\cR)$ on $\cB$,
then a  map $u$ is strictly $\CC$-proper if and only if for any base change $\bar u\to u$, the functor $\bar u^*$ preserves the class $\cL$.
In particular, if the factorization system $(\cL,\cR)$ is a modality, then every map is strictly proper.
\end{lemma}
\begin{proof}
A map $u$ is strictly proper if, for any base change $\bar u\to u$, the functor $\bar u_*$ preserves the maps in $\cR$, but this is equivalent to $\bar u^*$ preserving the class $\cL$.
And when $(\cL,\cR)$ is a modality, every $\bar u^*$ preserves the class $\cL$.
\end{proof}

\begin{definition}[Acyclic and localic maps]
\label{def:acyclic}
A map $u:X\to Y$ in $\cB$ is called {\it $\CC$-pre-acyclic}
if $u^*:\CC(Y)\to \CC(X)$ is an equivalence.
A map $u:X\to Y$ in $\cB$ is called {\it $\CC$-acyclic}
if every base change of $u$ is $\CC$-pre-acyclic.
We denote by $\Alpha(\CC)$ (Alpha) the codomain subfibration of $\CC$-acyclic maps. 
Acyclic maps are always both smooth and proper.
A map is called {\it $\CC$-localic} if it is right orthogonal to $\CC$-acyclic maps.
We denote by $\Lambda(\CC)$ the codomain subfibration of $\CC$-localic maps.
\end{definition}

The name `acyclic' is motivated by the following result.
The name `localic' is motivated by an application to topos theory (see third example of \cref{sec:topos}).
In a \oo topos $\cE$, an acyclic class is a class of maps containing all isomorphisms, closed under composition and base change, and under colimits in the arrow category of $\cE$ \cite[Definition 3.2.8]{ABFJ:HS}.
If $(\cL,\cR)$ is a (unique) factorization system on $\cE$, the class $\cL$ is acyclic if and only if $(\cL,\cR)$ is a modality.

\begin{lemma}
\label{lem:acyclic}
If $\cB$ is an \oo topos and if $\CC:\cB\op\to \CAT$ sends colimits to limits, 
then the class of $\CC$-acyclic maps is an acyclic class.
\end{lemma}
\begin{proof}
It is easy to see from the definition that $\Alpha(\CC)$ contains all isomorphisms, is closed under composition and base change.
It is also closed under small colimits in $\Arr \cB$ since the functor $\CC$ send colimits to limits and $\Alpha(\CC)$ is the inverse image of the class $\CAT^\simeq\subset \Arr\CAT$ which is closed under limits.
\end{proof}

The $\CC$-acyclic and $\CC$-localic maps sometimes form a factorization system on $\cB$. When this the case, it is always a modality since both classes are stable by base change by definition.

\section{Examples}
\label{sec:ex}

\subsection{Set theory and type theory}

The examples are summarized in \cref{table:set-theory}.
\medskip

\begin{table}[htbp]
\caption{Examples from set theory and type theory}
\label[table]{table:set-theory}
\medskip
\begin{center}
\renewcommand{\arraystretch}{1.2}
\begin{tabularx}{1.05\textwidth}{
|>{\hsize=.9\hsize\linewidth=\hsize\centering\arraybackslash}X
|>{\hsize=.9\hsize\linewidth=\hsize\centering\arraybackslash}X
|>{\hsize=1.1\hsize\linewidth=\hsize\centering\arraybackslash}X
|>{\hsize=1.1\hsize\linewidth=\hsize\centering\arraybackslash}X
|>{\hsize=1\hsize\linewidth=\hsize\centering\arraybackslash}X
|>{\hsize=1\hsize\linewidth=\hsize\centering\arraybackslash}X
|}
\hline
$\cB$ & $\CC$  & $\Sigma(\CC)$ & $\Pi(\CC)$ & $\Alpha(\CC)$& $\Lambda(\CC)$\\
\hline
\hline
Cat. of sets
& a category $C$
& sets $I$ for which $\coprod_I$ exists in $C$
& sets $I$ for which $\prod_I$ exists in $C$
& \multicolumn{2}{c|}{(see text)}
\\
\hline
Cat. of sets
& maps with $\kappa$-small fibers ($\kappa\geq 2$)
& maps with $\Sigma$-small fibers for $\Sigma = \sup \{\sigma\ |\ {\rho<\kappa} \Ra \rho.\sigma<\kappa\}$
& maps with $\Pi$-small fibers for $\Pi = \sup \{\pi\ |\ {\rho<\kappa} \Ra \rho^\pi<\kappa\}$
& bijections
& all maps
\\
\hline
Cat. of sets
& subsets ($\kappa=2$)
& all maps
& all maps
& bijections
& all maps
\\
\hline
$\cB$
& hyperdoctrine $\cB\op\to \Poset$
& all maps
& all maps
& ?
& ?\\
\hline
$\cB$
& codomain fibration
& all maps
& stably exp. maps
& isomorphisms
& all maps
\\
\hline
$\cB$
& $\pi$-clan structure $\DD\subset \BB$
& strict smooth = all maps in $\cD$
& strict proper $\supset$ all maps in $\cD$
& ?
& ?\\
\hline
A category $\cE$ with a subobject classifier $\Omega$
& dominance $\cO\subset\Omega$ 
& smooth = overt~maps 
strict~smooth = maps in $\cO$
& proper = proper~maps 
str.~proper = ?
& \multicolumn{2}{c|}{(see text)}
\\
\hline
A 1-topos $\cE$
& Grothendieck topology $\Omega_j\subset\Omega$ 
& smooth = all~maps 
strict~smooth = closed monos
& proper = ``quasi-compact'' maps
& maps inverted by the localization
& relative sheaves
\\
\hline
\hline
$\cS$, the \oo category of \oo groupoids
& subcategory $\cS\size\kappa$ of $\kappa$-small \oo groupoids
& strict : maps with $\kappa$-small fibers (if~$\kappa$~regular)
& strict : maps with $\kappa$-small fibers (if~$\kappa$~inacc.)
& ?
& ?
\\
\hline
%
\hline
An \oo topos $\cE$
& subuniverse $\TT_n$ of $n$-truncated objects
& the whole universe
(strict smooth maps are the $n$-truncated ones)
& the whole universe
(all maps are strictly proper)
& $(n+1)$\=/conn. maps
& $(n+1)$\=/trunc. maps
\\
\hline
An \oo topos $\cE$
& subuniverse $\RR\subset \EE$ of modal types (for a modality $(\cL,\cR)$ on $\cE$)
& the whole universe
(strict maps are those in $\RR$)
& ?
& décalage class of $\cL$ \cite{ABFJ:GT}
& right class of the décalage modality of \cite{ABFJ:AP}\\
\hline
An \oo topos $\cE$
& subuniverse $\FF\subset \EE$ of sheaves assoc. to a lex loc. $\cF=\cE[\cW^{-1}]$
& the whole universe 
(strict maps are those in $\FF$)
& ?
& maps in $\cW$
& maps in $\FF$
\\
\hline
\end{tabularx}
\end{center}
\end{table}

Any 1-category $C$ represents a functor $\cB\op=\Set\sfop\to \CAT$, sending $I$ to $C^I$.
In this setting, all conditions on maps can be computed fiberwise.
A set $I$ is smooth (proper) if the coproduct $\coprod_I:C^I\to C$
(product functor $\prod_I:C^I\to C$) exists.
Notice that in this case the Beck--Chevalley condition are always fulfilled since everything is determined fiberwise (\cref{rem:BC}).
A map is smooth (proper) if its fibers are smooth (proper) sets.
If $C$ is the terminal category, every set is smooth and proper.
If $C$ is the initial category, the smooth (proper) maps are the surjections.
A set $I$ is acyclic if $C\to C^I$ is an equivalence.
If $C$ is the terminal category, every set is acyclic (and every map is acyclic).
If $C$ is the empty category, the acyclic sets are the non-empty sets (the acyclic maps are surjections and the localic maps are the injections).
If $C$ is otherwise, the only acyclic sets are the singletons, thus the acyclic maps are the bijections and every map is localic.

Similar considerations hold when $\cB=\cS$ is the \oo category of spaces and $C$ some \oo category. 
In particular, \cref{rem:BC} applies and the Beck--Chevalley conditions are always fulfilled.

\medskip
The second example is the one detailled in the introduction of \cref{setting}.
It is an instance of the previous example if and only if $\kappa$ is a regular cardinal (use $C=\Set\size\kappa$).
The example of the fibration of subsets (or of injections), is the case $\kappa=2$.
Every map $u:I\to J$ is smooth and proper.
The functors $u_!$ and $u_*$ are the two direct images of subsets, classically related to existential and universal quantifiers ($u_!A = \{j\ |\ \exists a\in A, a\in u^{-1}(j)\}$, 
$u_*A = \{j\ |\ \forall a\in A, a\in u^{-1}(j)\}$).
More generally, the setting of sets and subsets could be replaced by a hyperdoctrine in the sense of Lawvere \cite{Lawvere:Adjoint,Lawvere:Eq, Seely:hyperdoctrines}.
We recall the correspondence between logical quantifiers and the adjoints to base change in \cref{table:logic}.

\begin{table}[htbp]
\caption{Quantifiers and direct images}
\label[table]{table:logic}
\medskip
\begin{center}
\renewcommand{\arraystretch}{1.2}
\begin{tabularx}{.9\textwidth}{
|>{\hsize=1.1\hsize\linewidth=\hsize\centering\arraybackslash}X
|>{\hsize=1\hsize\linewidth=\hsize\centering\arraybackslash}X
|>{\hsize=1\hsize\linewidth=\hsize\centering\arraybackslash}X
|>{\hsize=1.2\hsize\linewidth=\hsize\centering\arraybackslash}X
|>{\hsize=.8\hsize\linewidth=\hsize\centering\arraybackslash}X
|>{\hsize=.9\hsize\linewidth=\hsize\centering\arraybackslash}X
|}
\cline{2-6}
\multicolumn{1}{c|}{} 
& Indexing object & Families & change of index & left image & right image\\
\hline
Predicate logic & variables & predicates & substitution & $\exists$ & $\forall$ \\
\hline
Dependent type theory & contexts & dependent types & substitution & $\Sigma$ & $\Pi$ \\
\hline
Category theory:
fibration $C\to B$ & object in base & object in fiber & base change $u^*$ & $u_!$ & $u_*$\\
\hline
\end{tabularx}
\end{center}
\end{table}

\medskip
In the case of the codomain fibration, we have seen in \cref{prop:codomain-fib} that all maps are smooth and that proper maps are the exponentiable maps whose pullbacks are also exponentiable.
All the projections $X\times Y\to X$ are proper if and only if 
the category $\cB$ is cartesian closed, and all the maps are proper if and only if the category $\cB$ is locally cartesian closed.
In this setting, only the isomorphisms are acylic and every map is localic.

\medskip
Another example related to type theory is that of a category $\cB$ with a $\pi$-clan structure in the sense of Joyal \cite{Joyal:Clans}.
The clan structure distinguishes a class of maps in $\cD\subset\cB$ closed under base change and containing all isomorphisms, or equivalently a regular calibration $\DD\subset \BB$.
This implies that strict smooth maps are exactly the maps in $\cD$.
And the definition of the $\pi$-clan structure says that every map in $\cD$ is strictly proper \cite[Definition 2.4.1]{Joyal:Clans}.

\medskip
If $\cB$ is a cartesian closed 1-category with a subobject classifier (e.g. 1-topos), we consider the example of a {\it dominance}, which is a regular calibration of monomorphisms classified by a subobject $\cO\subset\Omega$ of the subobject classifier \cite{Escardo:Pittsburgh,Hyland:SDT,Rosolini:PhD}.
Intuitively, the object $\cO$ classifies the subobjects meant to be ``open'' in the sense of topology, and the exponential $\cO^A\subset\Omega^A$ is the lattice of open subobjects sitting within that of all subobjects.
The smooth (proper) maps define a notion of open (compact) maps in the sense of topology.
The strictly smooth maps are those classified by $\cO$ (\cref{lem:strict-smooth}).
The strictly proper maps are those maps for which the direct image can be computed within the posets of all subobjects.
Acyclic and localic maps are considered in the literature, but (seemingly) only with the base change property along cartesian projections and not along arbitrary maps.
They are called {\it $\cO$-equable} maps and {\it $\cO$-replete} maps in \cite{Hyland:SDT}.

When $\cB$ is a 1-topos, any Grothendieck topology defines a dominance $\Omega_j\subset \Omega$ where $\Omega_j$ is the classifier of closed monos.
In this case, all maps are smooth, the closed monos are the strictly smooth maps.
The acyclic maps are the maps inverted in the localization by the topology,
and the localic maps are the relative sheaves.
A map $A\to 1$ is proper if and only if $A$ it is quasi-compact (every covering family has a covering finite subfamily).
General proper maps can be described by a relative version of the same condition.

\medskip
The \oo category $\cS$ of \oo groupoids is a higher categorical model for dependent type theory with univalence: precisely, the small subuniverses correspond to univalent maps in $\cS$.
It seems difficult to describe the smooth, proper, acyclic and localic maps for an arbitrary subuniverse, but it is easier with strict smooth and proper maps.
We shall only consider the subuniverse of $\kappa$-small spaces.
When $\kappa$ is regular, the strict smooth maps are the maps with $\kappa$-small fibers by \cref{lem:strict-smooth}.
When $\kappa$ is inaccessible, the strict proper maps are the maps with $\kappa$-small fibers by \cref{lem:strict-proper}.
In particular, if $\kappa$ is the inaccessible cardinal bounding the size of small objects, all maps are smooth and proper.
In this case, the acyclic maps are reduced to equivalences and all maps are localic.

\medskip
In the example of an \oo topos $\cE$, we denote the universe by $\EE$.%
\footnote{The same example could be presented in the setting of 1-topoi, but the formalism of \oo topoi is just easier.}
By definition of an \oo topos, the functor $\EE:\cE\op\to \CAT$ sends colimits of $\cE$ to limits in $\CAT$, 
and the interesting subuniverses are those satisfying a similar condition (corresponding to local classes of maps \cite[Proposition 6.1.3.7]{Lurie:HTT}).
For $n\geq -2$, we denote by $\cC_n$ and $\cT_n$ the classes of $n$-connected and $n$-truncated maps.
For example, the class $\cC_{-1}$ and $\cT_{-1}$ are the classes of surjections and monomorphisms.
Every map can be factored uniquely into an $n$-connected map followed by an $n$-truncated maps, and this factorization is stable under base change (it is a modality).
The class $\cT_n$ contains all isomorphisms, is closed under composition, and is local. 
It defines a regular calibration $\TT_n\subset\EE$ which is also a reflective subuniverse (where the reflection is given by the factorization).
Every map is smooth for $\TT$, and the strict smooth maps are those in $\TT$.
Every map is strictly proper (hence proper) for $\TT$ because dependent products preserve the truncation level of objects.
The $\TT_n$-acyclic maps and $\TT_n$-localic maps can be characterized as the $(n+1)$-connected maps and $(n+1)$-truncated maps (see below).

More generally, a modality on $\cE$ is a (unique) factorization system $(\cL,\cR)$ on $\cE$ such that both classes are stable under base change.
Then both classes $\cL$ and $\cR$ are local by \cite[Proposition 3.6.5]{ABFJ:GBM} and define 
subuniverses $\LL,\RR\subset\EE:\cE\op\to \CAT$.
We shall consider the case $\CC=\RR$.
Since the class $\cR$ is closed under composition and base change, \cref{lem:strict-smooth} shows that the strict smooth maps are all maps in $\cR$.
But the subuniverse actually admits sums indexed by all maps in $\cE$.
Given a map $u:X\to Y$ in $\cE$ and a map $r:\bar X\to X$ in $\cR$, the $(\cL,\cR)$-factorization of the composite map $\bar X\to X\to Y$ gives maps $\bar X\to \bar Y$ in $\cL$ and $r':\bar Y\to Y$ in $\cR$.
The image of $r$ by $u_!$ is the map $\bar Y\to Y$.
The Beck--Chevalley condition holds because the factorization is stable under base change.
A more conceptual way to see this is to say that $\RR$ is in fact a reflective subuniverse (where the reflection is given by the factorization) and therefore complete under all sums existing in $\EE$.

In this example, it seems difficult to describe the proper maps without further assumption on $\RR$.
However, the acyclic maps are exactly the maps in the {\it décalage} of the class $\cL$ \cite[2.2.7]{ABFJ:GT} (they are also the `fiberwise $\RR$-equivalences' of \cite[Def.3.3.1]{ABFJ:HS}).
In particular, they form an acyclic class in $\cE$ \cite[Theorem 3.3.9]{ABFJ:HS}.
In this context, acyclic and localic maps define a modality on $\cE$ which is detailled in \cite[Theorem 2.3.32]{ABFJ:AP}.

Finally, any left-exact localization of \oo topoi  $q^*:\cE\to \cE[\cW^{-1}]$ provide a modality $(\cW,\cF)$ on $\cE$, where $\cW$ is the class if maps inverted by $q^*$ and $\cF$ is the class of ``relative sheaves'' \cite{RSS,ABFJ:HS}.
Such a modality is left-exact in the sense that the factorization of a map preserves finite limits (in the arrow category).
In fact, there is a bijection between left-exact localizations and left-exact modalities \cite{ABFJ:HS}.
Let $\WW\subset\EE\supset\FF$ be the corresponding subuniverses.
If $\EE':\cE[\cW^{-1}]\op\to \CAT$ is the universe of the \oo topos $\cE[\cW^{-1}]$, one can show that $\FF=\EE'\circ q^*:\cE\op\to \CAT$.
Since left-exact modalities are fixed by décalage \cite[Lemma 2.4.6~(1)]{ABFJ:GT}, the acyclic (localic) maps coincides with the class $\cW$ ($\cF$).

\subsection{Category theory}

We now consider examples from category theory.
The study of fibrations of categories was the motivation of Grothendieck for introducing his notion of smooth and proper functors.
The examples are summarized in \cref{table:category}.

\begin{table}[htbp]
\caption{Examples from category theory}
\label[table]{table:category}
\medskip
\begin{center}
\renewcommand{\arraystretch}{1.3}
\begin{tabularx}{.9\textwidth}{
|>{\hsize=.8\hsize\linewidth=\hsize\centering\arraybackslash}X
|>{\hsize=.9\hsize\linewidth=\hsize\centering\arraybackslash}X
|>{\hsize=1.1\hsize\linewidth=\hsize\centering\arraybackslash}X
|>{\hsize=1.2\hsize\linewidth=\hsize\centering\arraybackslash}X
|}
\hline
$\cB$ & $\CC$  & $\Sigma(\CC)$ & $\Pi(\CC)$ 
\\
\hline
\hline
Cat. of categories
& all functors
& all functors
& Conduch\'e fib.
\\
\hline
Cat. of categories
& left fibrations
& smooth = ?
strictly~smooth = left~fib.
& proper = ?
strictly~proper = Lurie-proper functor (=~Grothendieck-smooth)
\\
\hline
Cat. of categories
& right fibrations
& smooth = ? 
strictly~smooth  = right~fib.
& proper = ?
strictly~proper = Lurie-smooth functors (=~Grothendieck-proper)
\\
\hline
Cat. of categories
& cocartesian fibrations
& smooth = ?
strictly~smooth = cocart.~fib. 
& proper = ?
strictly~proper = Cond. + $u^*$ pres. ff left adj.
\\
\hline
Cat. of categories
& cartesian fibrations
& smooth = ? 
strictly~smooth = cart.~fib.
& proper = ?
strictly~proper = Cond. + $u^*$ pres. ff right adj.
\\
\hline
\end{tabularx}
\end{center}
\end{table}


\medskip
The first example is that of the universe of $\cB=\Cat$.
It will be the base of all the other examples.
We denote by $\CCat$ the universe of $\Cat$ (which can be thought of either as the codomain fibration, or as the slice functor $C\mapsto \Cat\slice C$).
The category $\Cat$ is cartesian closed but not locally cartesian closed.
The exponentiable functors are the Conduch\'e fibrations (see \cite[Lemma 1.11]{Ayala-Francis:fibrations} for a higher categorical account).
The acyclic maps are the equivalences of categories and every functor is localic.

\medskip
The following examples will study various kinds of ``fibrations'' condition on functors.
This will lead to consider functors $\CC:\Cat\sfop \to \CAT$ with a natural forgetful morphism $\CC\to \CCat$ into the universe of $\Cat$ (or equivalently a morphism between the corresponding fibrations over $\Cat$).

Let us start with the the functor $\LFib:\Cat\sfop\to \CAT$ sending a category $C$ to its category $\LFib(C) = \fun C\cS$ of (small) left fibrations.
This functor is (up to size issues) representable by the category $\cS$.
The Grothendieck construction provides a natural transformation $\LFib(C) \to \Cat\slice C$ and we can talk about strict notions.
Recall that the left fibrations are the right class of a (unique) factorization system on $\Cat$, where the left class is that of initial functors.
The characterization of smooth and proper maps is open.
The previous factorization system is not stable under base change and we cannot apply \cref{lem:strict-smooth-modality}.
But the strict smooth maps are exactly the left fibrations themselves.
The class of proper functors contains right fibrations 
And we can use \cref{lem:strict-proper} to characterize the strict proper maps as the exponentiable functors $u:C\to D$ such that for any base change $\bar u\to u$, the pullback along $\bar u$ preserves the class of initial functors.
This is exactly the notion of proper functor of \cite[dual of Definition 4.1.2.9]{Lurie:HTT} and the notion of smooth functor of Grothendieck in Pursuing Stacks (and also in \cite[21.1]{Joyal:QC} and \cite[dual of Definition 4.4.1]{Cisinski:book}).%
\footnote{The reversal of names is due to a preference for covariant or contravariant functors in the definitions. Grothendieck's convention uses presheaves, we have preferred to use covariant functors.}
In particular, right fibrations are strict proper functors for $\LFib$.
The previous characterization of strict proper functors show that they are also left-final functors in the sense of Ayala--Francis since they satisfy the condition of \cite[Lemma 4.1.1~(a)]{Ayala-Francis:fibrations}.
However, the condition to be left-final seems strictly weaker.
In the dual situation of right fibrations, the notion of smooth and proper are reversed.

\medskip
Finally, we can look at the example of (small) cocartesian (and cartesian) fibrations.
The functor of interest is $C\mapsto \CcFib(C) = \fun C \Cat$. 
Up to size issues, it is represented by $\Cat$ itself.
The class of cocartesian fibrations is almost the right class of a weak factorization system on $\Cat$ (in 1-categories, it is the right class of an algebraic factorization system \cite{Bourke-Garner}).
The corresponding left class is the class of fully faithful left adjoint functor.
Any functor $f:C\to D$ admits a factorization $C\to f\downarrow D\to D$ into a fully faithful left adjoint functor followed by a cocartesian fibration (where $f\downarrow D$ if the fiber product of $C\to D \xot {dom} \Arr D$).
Again, the smooth and proper maps are difficult to find.
The strict smooth maps are the cocartesian fibrations themselves.
The previous factorization is enough to apply \cref{lem:strict-proper} to characterize the strictly proper maps as the exponentiable functors $u:C\to D$ such that $u^*$ preserves fully faithful left adjoint functors. 
This class includes that of cartesian fibrations.
In the dual situation of cartesian fibrations, the notion of smooth and proper are reversed.

\medskip
In all the ``fibration'' examples, the acyclic maps are the equivalences and every functor is localic.
To see this, recall that a functor $u:C\to D$ induces an equivalence $u_!:\cS^C\rightleftarrows \cS^D:u^*$ if and only if it is a Morita equivalence.
The acyclic maps are Morita equivalences which are stable under base change.
This implies that they must be (essentially) surjective functors.
Since $u_!$ is an equivalence, $u$ is fully faithful, and is therefore an actual equivalence.

\subsection{Topology and geometry}
\label{sec:topos}
We now consider examples where the base category $\cB$ is a category of topological or geometrical objects.
We are going to see a tight connection between direct images functors
and classical topological conditions.
The examples are summarized in \cref{table:topology}.

\begin{table}[htbp]
\caption{Examples from topology and geometry}
\label[table]{table:topology}
\medskip
\begin{center}
\renewcommand{\arraystretch}{1.2}
\begin{tabularx}{\textwidth}{
|>{\hsize=.7\hsize\linewidth=\hsize\centering\arraybackslash}X
|>{\hsize=1.2\hsize\linewidth=\hsize\centering\arraybackslash}X
|>{\hsize=1.1\hsize\linewidth=\hsize\centering\arraybackslash}X
|>{\hsize=1\hsize\linewidth=\hsize\centering\arraybackslash}X
|>{\hsize=1\hsize\linewidth=\hsize\centering\arraybackslash}X
|>{\hsize=1\hsize\linewidth=\hsize\centering\arraybackslash}X
|}
\hline
$\cB$ & $\CC$  & $\Sigma(\CC)$ & $\Pi(\CC)$ & $\Alpha(\CC)$& $\Lambda(\CC)$\\
\hline
\hline
Cat. of locales
& all maps
& all maps
& exponentiable maps
& homeo\-morphisms
& all maps\\
\hline
Cat. of locales
& open immersions (rep. by Sierpi\'nski space)
& open maps (strict: open immersions)
& proper maps (strict: locally compact and proper)
& homeo\-morphisms
& all maps\\
\hline
\hline
Cat. of 1-topoi
& open immersions (rep. by Sierpi\'nski topos)
& open maps \cite{Johnstone:Elephant}
& proper maps \cite{Moerdijk-Vermeulen} (strict: exponentiable + proper)
& hyper\-connected morphisms
& localic morphisms\\
\hline
Cat. of 1-topoi
& étale maps (rep. by object classifier)
& locally connected maps \cite{Johnstone:Elephant}
& tidy maps \cite{Moerdijk-Vermeulen}
& equivalences 
& all morphisms\\
\hline
\hline
Cat. of \oo topoi
& étale maps (rep. by object classifier)
& locally contractible maps \cite{Martini-Wolf:internal}
& proper maps \cite{Martini-Wolf:proper}
& equivalences 
& all morphisms\\
\hline
Cat. of \oo topoi
& $n$-tr. étale maps (rep. by $n$-truncated object classifier)
& locally $n$-connected maps (by methods similar to \cite{Martini-Wolf:internal})
& $n$-proper maps (by methods similar to \cite{Martini-Wolf:proper})
& hyper-${(n+1)}$\=/connected morphisms
& ${(n+1)}$-localic morphisms \\
\hline
\hline
Cat. of schemes
& torsion étale sheaves
& smooth map $\supset$ smooth morphisms \cite[Thm~7.3]{Freitag-Kiehl}
& proper map $\supset$ proper morphisms \cite[Thm~6.1]{Freitag-Kiehl}
& ?
& ?\\
\hline
\end{tabularx}
\end{center}
\end{table}

\medskip
In the first example, $\cB=\Locale$ is the category of locales%
\footnote{Using locales instead of topological spaces simplify the computations. Similar considerations are true for sober spaces.}
with the codomain fibration.
Every map is smooth, and the proper maps are the exponentiable ones.
The exponentiable locales are the locally compact ones (i.e. the locales such that $\Op X$ is a continuous lattice) \cite[Theorem 4.11]{Johnstone:Stone}.
Acyclic maps are the isomorphisms, and every map is localic.

The next example is that of the functor $\Locale\sfop\to \CAT$ sending a locale $X$ to its frame $\Op X$ of open domains.
This functor is represented by the Sierpi\'nski space.
Let us say that a map is an {\it open immersion} if it is isomorphism to the inclusion of an open.
The fibration corresponding to $\Op-$ is the subfibration of the codomain fibration of $\Locale$ spanned by open immersions.
This shows that the notion of strict smooth and proper maps makes sense.

The smooth morphisms are almost by definition the open morphisms of topology (this can be shown directly, or deduced from the similar result for topoi \cite[Lemma C.3.1.10]{Johnstone:Elephant}, see also \cite{Escardo:topology}).
By \cref{lem:strict-smooth} the strict smooth maps are the open immersions that are also open morphisms, but that is the case of every open immersion.
The proper morphisms are the proper morphisms of topology (i.e. the universally closed morphisms, this can be seen directly by considering left adjoints on the posets $\Op X\op$, or deduced from topos theory \cite[Lemma C.3.2.81]{Johnstone:Elephant}, see also \cite{Escardo:topology,Escardo:compact}).
The characterization of strict proper maps is open.
An open immersion is proper if and only if it is isomorphic to the inclusion of a clopen.
The computation of acyclic and localic maps is straighforward.

\medskip
In the third example, $\cB$ is the category of 1-topoi and geometric morphisms \cite{Johnstone:Elephant, Anel-Joyal:topo-logie}.
If we look at the codomain fibration, the only non-trivial class of maps is that of exponentiable maps \cite{JJ}.

Another important fibration is the one of open sub-1-topoi (or open immersions), sending a 1-topos $\tX$ to the poset of subterminal objects in its category of sheaves $\Sh\tX$ (its dual \emph{1-logos} in the sense of \cite{Anel-Joyal:topo-logie}).
This fibration is represented by the Sierpi\'nski topos (dual to the 1-logos $\Arr\Set$).
The identification of smooth maps as open morphisms of 1-topoi is done in \cite[Theorem C.3.1.28]{Johnstone:Elephant}, 
and that of proper maps as proper morphisms of 1-topoi is done in \cite[Corollary I.5.9]{Moerdijk-Vermeulen} and in \cite[Theorem C.3.1.28]{Johnstone:Elephant}.
Both references use a ``weak'' version of Beck--Chevalley conditions (where the mate transformation is only monic) for arbitrary sheaves. 
But restricted to subterminal objects this condition recovers the one from \cref{def:smooth} and this can be shown to be equivalent to the weak condition (see \cite[Proposition A.4.1.17]{Johnstone:Elephant} and \cite[Proposition 3.2]{Moerdijk-Vermeulen}).

Open immersions can be composed and define a regular calibration on $\Topos$.
Then, \cref{lem:strict-smooth} shows that the strict smooth maps are exactly the open immersion of 1-topoi.
The characterization of strict proper maps is open.
An open immersion of 1-topoi is proper if it is also the inclusion of a closed sub-1-topoi (corresponding to a decidable subterminal object).
Interestingly, in this example the notion of acyclic and localic maps recover the hyperconnected and localic morphisms of topoi (this is in fact the motivation for the name \emph{localic}).
The proof is straightforward for hyperconnected maps, and the fact that the right orthogonal to hyperconnected morphisms are the localic morphisms is \cite[Lemma A.4.6.4]{Johnstone:Elephant}.

\medskip
The next example is the case of the fibration of étale maps over the category of 1-topoi, sending a 1-topos $\tX$ to its category of sheaves of sets (its dual 1-logos) $\Sh \tX$.
This fibration is representable by the ``object classifier'' or the ``topos line'' (that we shall denote $\tA$, its 1-logos of sheaves is $\fun {\textsf{finset}}\Set$) and the existence of sums and products for the étale fibration can be interpreted as sums and products structures on the 1-topos $\tA$.
In this case, the smooth maps are the locally connected morphisms of 1-topoi \cite[Corollary C.3.3.16]{Johnstone:Elephant}, 
and the proper maps are the tidy morphisms of 1-topoi, see \cite[Corollary III.4.9]{Moerdijk-Vermeulen} or \cite[Corollary C.3.4.11]{Johnstone:Elephant}.
The strict smooth maps are all the étale morphisms.
The characterization of strict proper maps is open.
In this setting, the acyclic maps are the equivalences and every map is localic.

Interestingly, not every étale map is proper: the intersection of the two classes is the class of  finite maps.
This means that the ``topos line'' $\tA$ does not have arbitrary products indexed by itself, but only finite products.
This might be surprising since categories of sheaves are always locally cartesian closed, but the dependent products are not preserved by geometric morphisms, only finite products are.

\medskip
The generalization of the previous results to \oo topoi is quite rich!
For the case of the codomain fibration, exponentiable \oo topoi have been characterized in \cite[Theorem 21.1.6.12]{Lurie:SAG} and \cite{Anel-Lejay:topos-exp}.
The case of the fibration of étale morphisms has been studied recently by Martini and Wolf.
If $\cB=\Topos_\infty$ is the category of \oo topoi and $\CC(\tX) = \iSh \tX$ is the category of higher sheaves (the \oo logos dual to $\tX$),
then the smooth maps are the locally contractible morphisms of \oo topoi \cite{Martini-Wolf:internal},
and the proper maps are the proper morphisms of \oo topoi \cite{Lurie:HTT,Martini-Wolf:proper}.
The acyclic maps are equivalences and every map is localic.

It is interesting to restrict the fibration of higher sheaves (aka higher etale morphisms) to {\it $n$-truncated} sheaves only (the case $n=-1$ recovers the fibration of open immersions).
In this case, it is easy to adapt the results of Martini and Wolf to characterize 
the smooth maps as some {\it locally $n$-connected} morphisms of \oo topoi%
\footnote{morphisms $u$ for which the inverse image $u^*$ has a (local) left adjoint $u_!$ when restricted to the subcategories of $n$-truncated sheaves.}
and the proper maps are the {\it $n$-proper} morphisms of \oo topoi%
\footnote{morphisms $u$ whose direct image $u_*$ preserve (internal) filtered colimits of $n$-truncated objects.}
The strict smooth maps are the $n$-truncated étale morphisms of topoi.
The corresponding acyclic and localic maps define notions of hyper-$(n+1)$-connected morphisms and $(n+1)$-localic morphisms, and every geometric morphism should factors into a hyper-$(n+1)$-connected morphism followed by a $(n+1)$-localic morphism.
Moreover, an \oo topos $\tX$ should be $n$-localic in the sense of \cite[Definition 6.4.5.8]{Lurie:HTT} if and only if the morphism $\tX\to 1$ is $n$-localic in the previous sense.

\'Etale maps define a regular calibration of the universe of $\Topos_\infty$ and the notions of strictly smooth and strictly proper maps can be defined.
By \cref{lem:strict-smooth}, the strict smooth maps are exactly the etale maps.
The strict proper maps are exponentiable maps $u$ whose direct image $u_*$ preserves etale maps.
Their characterization is open.

\medskip
The last example of \cref{table:topology} is in algebraic geometry, where the base category is the category $\mathsf{Sch}$ of (noetherian) schemes, and the functor $\mathsf{Sch}\sfop \to \CAT$ is the one sending a scheme $X$ to its category of torsion étale sheaves \cite{Freitag-Kiehl}.
The characterization of smooth and proper maps in general is open, but \cite[Theorems 6.1 \& 7.3]{Freitag-Kiehl} show 
that smooth morphisms of schemes are smooth maps (if the torsion is prime to the characteristic), and 
that proper morphisms of schemes are proper maps.
This setting is the one that inspired Grothendieck to name his notions of smooth and proper functors \cite{Maltsiniotis:aspheric}.
The characterization of acyclic and localic maps is an open question.



\begin{thebibliography}{ABFJ24}

\bibitem[ABFJ20]{ABFJ:GBM}
M.~Anel, G.~Biedermann, E.~Finster, and A.~Joyal, \emph{{A generalized
  Blakers--Massey theorem}}, Journal of Topology \textbf{13} (2020), no.~4,
  1521--1553,
  \href{https://doi.org/10.1112/topo.12163}{doi:10.1112/topo.12163}.

\bibitem[ABFJ22]{ABFJ:HS}
\bysame, \emph{{Left-exact localizations of $\infty$-topoi I: Higher sheaves}},
  Advances in Mathematics \textbf{400} (2022), 108268,
  \href{https://doi.org/10.1016/j.aim.2022.108268}{doi:10.1016/j.aim.2022.108268}.

\bibitem[ABFJ23]{ABFJ:AP}
\bysame, \emph{{Left-exact localizations of $\infty$-topoi III: Acyclic
  Product}}, preprint (2023),
  \href{https://arxiv.org/abs/2308.15573}{arXiv:2308.15573}.

\bibitem[ABFJ24]{ABFJ:GT}
\bysame, \emph{{Left-exact localizations of $\infty$-topoi II: Grothendieck
  topologies}}, Journal of Pure and Applied Algebra \textbf{228} (2024), no.~3,
  107472,
  \href{https://doi.org/10.1016/j.jpaa.2023.107472}{doi:10.1016/j.jpaa.2023.107472}.

\bibitem[AF20]{Ayala-Francis:fibrations}
D.~Ayala and J.~Francis, \emph{Fibrations of $\infty$-categories}, {Higher
  Structures} \textbf{4} (2020), no.~1, 168--265.

\bibitem[AJ21]{Anel-Joyal:topo-logie}
M.~Anel and A.~Joyal, \emph{Topo-logie}, New Spaces in Mathematics: Formal and
  Conceptual Reflections (M.~Anel and G.~Catren, eds.), Cambridge University
  Press, 2021, pp.~155--257.

\bibitem[AL19]{Anel-Lejay:topos-exp}
M.~Anel and D.~Lejay, \emph{Exponentiable $\infty$-topoi (version 2)}, Preprint
  (2019),
  \href{http://mathieu.anel.free.fr/mat/doc/Anel-Lejay-Exponentiable-topoi.pdf}{on
  Anel's homepage}.

\bibitem[B{\'e}n75]{Benabou:calibration}
J.~B{\'e}nabou, \emph{Th{\'e}ories relatives {\`a} un corpus}, C. R. Acad. Sc.
  Paris \textbf{281} (1975), 831--834.

\bibitem[BG14]{Bourke-Garner}
J.~Bourke and R.~Garner, \emph{Algebraic weak factorisation systems {I}:
  {A}ccessible {AWFS}}, Journal of Pure and Applied Algebra \textbf{220}
  (2014), 108--147.

\bibitem[BW23]{Buchholtz-Weinberger}
U.~Buchholtz and J.~Weinberger, \emph{Synthetic fibered $(\infty,1)$-category
  theory}, Higher Structures \textbf{7} (2023), no.~1, 74--165.

\bibitem[Cis19]{Cisinski:book}
D.-C. Cisinski, \emph{Higher categories and homotopical algebra}, Cambridge
  Studies in Advanced Mathematics, Cambridge University Press, 2019.

\bibitem[Esc04a]{Escardo:topology}
M.~Escard{\'o}, \emph{Synthetic topology: of data types and classical spaces},
  Electronic Notes in Theoretical Computer Science \textbf{87} (2004), 21--156,
  Proceedings of the Workshop on Domain Theoretic Methods for Probabilistic
  Processes.

\bibitem[Esc04b]{Escardo:Pittsburgh}
\bysame, \emph{Topology via higher-order intuitionistic logic}, Lecture Notes
  at MFPS in Pittsburgh, 2004,
  \href{https://www.cs.bham.ac.uk/~mhe/papers/pittsburgh.pdf}{On Escard{\'o}'s
  website}.

\bibitem[Esc20]{Escardo:compact}
\bysame, \emph{Intersections of compactly many open sets are open}, 2020,
  \href{https://arxiv.org/abs/2001.06050}{arXiv:2001.06050}.

\bibitem[FK88]{Freitag-Kiehl}
E.~Freitag and R.~Kiehl, \emph{Etale cohomology and the {W}eil conjecture},
  Ergebnisse der {M}athematik und ihrer {G}renzgebiete: {A} series of modern
  surveys in mathematics. Folge 3, no. vol.~13, Springer-Verlag, 1988.

\bibitem[Hyl91]{Hyland:SDT}
J.~M.~E. Hyland, \emph{First steps in synthetic domain theory}, Category Theory
  (Berlin, Heidelberg) (Aurelio Carboni, Maria~Cristina Pedicchio, and Guiseppe
  Rosolini, eds.), Springer Berlin Heidelberg, 1991, pp.~131--156.

\bibitem[JJ82]{JJ}
P.~T. Johnstone and A.~Joyal, \emph{Continuous categories and exponentiable
  toposes}, Journal of Pure and Applied Algebra \textbf{25} (1982), no.~3, 255
  -- 296.

\bibitem[Joh82]{Johnstone:Stone}
P.~T. Johnstone, \emph{Stone spaces}, Cambridge Studies in Advanced
  Mathematics, Cambridge University Press, 1982.

\bibitem[Joh02]{Johnstone:Elephant}
\bysame, \emph{Sketches of an elephant: a topos theory compendium. {V}ol. 1 and
  2}, Oxford Logic Guides, Oxford University Press, Oxford, 2002.

\bibitem[Joy08]{Joyal:QC}
A.~Joyal, \emph{{Notes on Quasicategories}},
  {\href{http://www.math.uchicago.edu/~may/IMA/Joyal.pdf}{On P. May's
  homepage}}.

\bibitem[Joy17]{Joyal:Clans}
\bysame, \emph{Notes on clans and tribes},
  \href{https://arxiv.org/abs/1710.10238}{arXiv:1710.10238}.

\bibitem[Law69]{Lawvere:Adjoint}
F.W. Lawvere, \emph{{A}djointness in {F}oundations}, Dialectica \textbf{23}
  (1969), 281--296, Reprints in Theory and Applications of Categories, No. 16
  (2006), pp 1-16 (revised 2006-10-30).

\bibitem[Law70]{Lawvere:Eq}
\bysame, \emph{Equality in hyperdoctrines and comprehension schema as an
  adjoint functor}, Applications of Categorical Algebra \textbf{17} (1970),
  1--14.

\bibitem[Lur09]{Lurie:HTT}
J.~Lurie, \emph{Higher topos theory}, Annals of Mathematics Studies, vol. 170,
  Princeton University Press, Princeton, NJ, 2009.

\bibitem[Lur17]{Lurie:SAG}
\bysame, \emph{Spectral algebraic geometry}, version of February 3, 2017,
  \href{https://www.math.ias.edu/~lurie/papers/SAG-rootfile.pdf}{Online book}.

\bibitem[Mal05]{Maltsiniotis:aspheric}
G.~Maltsiniotis, \emph{Structures d{\textquoteright}asph\'ericit\'e, foncteurs
  lisses, et fibrations}, Annales math\'ematiques Blaise Pascal \textbf{12}
  (2005), no.~1, 1--39 (fr). \MR{2126440}

\bibitem[Mar22]{Martini:cocart}
L.~Martini, \emph{Cocartesian fibrations and straightening internal to an  $\infty$-topos}, 2022,
\href{https://arxiv.org/abs/2204.00295}{arXiv:2204.00295}

\bibitem[MW23a]{Martini-Wolf:internal}
L.~Martini and S.~Wolf, \emph{Internal higher topos theory}, 2023,
  \href{https://arxiv.org/abs/2303.06437}{arXiv:2303.06437}.

\bibitem[MW23b]{Martini-Wolf:proper}
\bysame, \emph{Proper morphisms of $\infty$-topoi}, 2023,
  \href{https://arxiv.org/abs/2311.08051}{arXiv:2311.08051}.

\bibitem[Moe82]{Moens:PhD}
J.-L. Moens, \emph{Caract{\'e}risation des topos de faisceaux sur un site interne
  {\`a} un topos}, Ph.D. thesis, UCL-Universit{\'e} Catholique de Louvain,
  1982.

\bibitem[MV00]{Moerdijk-Vermeulen}
I.~Moerdijk and J.~J.~C.~Vermeulen, \emph{Proper maps of toposes}, Memoirs of the
  American Mathematical Society \textbf{148} (2000).





\bibitem[Rie23a]{Riehl:Notices}
E.~Riehl, \emph{Could $\infty$-category theory be taught to undergraduates?},
  Notices of the American Mathematical Society \textbf{70} (2023), no.~5.

\bibitem[Rie23b]{Riehl:Semantics}
\bysame, \emph{On the $\infty$-topos semantics of homotopy type theory}, 2023,
  Lecture notes for a mini-course at CIRM, Luminy, Feb 2022.

\bibitem[Ros86]{Rosolini:PhD}
G.~Rosolini, \emph{Continuity and effectiveness in topoi}, PhD thesis, University of Oxford.


\bibitem[RS17]{Riehl-Shulman}
E.~Riehl and M.~Shulman, \emph{A type theory for synthetic
  $\infty$-categories}, Higher Structures \textbf{1} (2017), no.~1, 116--193.

\bibitem[RSS19]{RSS}
E.~Rijke, M.~Shulman, and B.~Spitters, \emph{Modalities in homotopy type
  theory}, Logical Methods in Computer Science \textbf{16} (2019), no.~1.

\bibitem[See83]{Seely:hyperdoctrines}
R.A.~G. Seely, \emph{{H}yperdoctrines, {N}atural {D}eduction and the {B}eck
  {C}ondition}, Math. Log. Q. \textbf{29} (1983), 505--542.

\bibitem[Shu19]{Shulman:topos}
M.~Shulman, \emph{All $(\infty,1)$-toposes have strict univalent universes},
  \href{https://arxiv.org/abs/1904.07004}{arXiv:1904.07004}.

\bibitem[Str23]{Streicher:fibrations}
T.~Streicher, \emph{{Fibered Categories à la B{\'e}nabou}}, Lecture notes.
  \href{https://arxiv.org/abs/1801.02927}{arXiv:1801.02927}.

\bibitem[Wei22a]{Weinberger:IntlSums}
J.~Weinberger, \emph{Internal sums for synthetic fibered
  $(\infty,1)$-categories}, 2024,
 Journal of Pure and Applied Algebra, Volume 228, Issue 9, 107659, \href{https://doi.org/10.1016/j.jpaa.2024.107659}{doi:10.1016/j.jpaa.2024.107659}

\bibitem[Wei22b]{Weinberger:StrExt}
\bysame, \emph{Strict stability of extension types}, 2022,
  \href{https://arxiv.org/abs/2203.07194}{arXiv:2203.07194}.

\end{thebibliography}

\providecommand{\bysame}{\leavevmode\hbox to3em{\hrulefill}\thinspace}
\providecommand{\MR}{\relax\ifhmode\unskip\space\fi MR }
\providecommand{\MRhref}[2]{%
  \href{http://www.ams.org/mathscinet-getitem?mr=#1}{#2}
}

\end{document}